\documentclass{amsart}[12pt]

\usepackage{amsmath, amsthm, amscd, amsfonts,}
\usepackage{graphicx,float}
\usepackage{amssymb}
\usepackage[cmtip,arrow]{xy}
\usepackage{pb-diagram,pb-xy}

\DeclareMathOperator{\dom}{dom}

\DeclareMathOperator{\ran}{ran}

\def\MRB{{\mathbb{R}}}
\def\MCB{{\mathbb{C}}}

\def\lan{\langle}
\def\ran{\rangle}
\def\a{\alpha}

\def\ov{\overline}

\newtheorem{theorem}{Theorem}[section]
\newtheorem{lemma}[theorem]{Lemma}

\numberwithin{equation}{section}

\usepackage{pb-diagram}

\def\lan{\langle}
\def\ran{\rangle}
\def\ov{\bar}
\def\rmark{\mbox{$\rm\bf\rule{0.06em}{1.45ex}\kern-0.05em R$}}
\def\pmark{\mbox{$\rm\bf\rule{0.06em}{1.45ex}\kern-0.05em P$}}
\def\nmark{\mbox{$\rm\bf\rule{0.06em}{1.45ex}\kern-0.05em N$}}
\def\vdash{\mbox{$\rm\| \kern-0.13em -$}}

\def\k{\kappa}

\def\lan{\langle}
\def\ran{\rangle}
\def\a{\alpha}

\def\ov{\overline}

\def\lan{\langle}
\def\ran{\rangle}
\def\ov{\bar}
\def\rmark{\mbox{$\rm\bf\rule{0.06em}{1.45ex}\kern-0.05em R$}}
\def\pmark{\mbox{$\rm\bf\rule{0.06em}{1.45ex}\kern-0.05em P$}}
\def\nmark{\mbox{$\rm\bf\rule{0.06em}{1.45ex}\kern-0.05em N$}}
\def\vdash{\mbox{$\rm\| \kern-0.13em -$}}
\newcommand{\lusim}[1]{\smash{\underset{\raisebox{1.2pt}[0cm][0cm]{$\sim$}}
{{#1}}}}

\title[Adding many random reals may add many Cohen reals]{Adding many random reals may add many Cohen reals}

\author[M. Golshani]{Mohammad Golshani}

\date{}

\thanks{The  author's research has been supported by a grant from IPM (No. 95030417). He also thanks Moti Gitik for his useful comments and suggestions.}

\begin{document}
\begin{abstract}
Let $\kappa$ be an infinite cardinal. Then, forcing with $\MRB(\kappa) \times \MRB(\k)$ adds a generic filter for $\MCB(\kappa);$
where $\MRB(\kappa)$ and $\MCB(\kappa)$ are the forcing notions for adding $\k$-many random reals and adding $\k$-many Cohen reals respectively.
\end{abstract}

\thanks{ } \maketitle

\section{Introduction}
For a cardinal $\k >0$ let $\MRB(\kappa)$ be the forcing notion for adding $\k$-many random reals and let $\MCB(\kappa)$ be the Cohen forcing
 for adding $\k$-many Cohen reals\footnote{See Section \ref{forcings} for the definition of the forcing notions $\MRB(\k)$ and $\MCB(\k)$.}.

It is a well-know fact that forcing with $\MRB(1) \times \MRB(1)$ adds a Cohen real; in fact, if $r_1, r_2$
are the added random reals, then $c=r_1 + r_2$ is Cohen \cite{judah}. This in turn implies all reals $c+a,$ where $a \in \MRB^V,$ are Cohen,
and so, we have continuum many Cohen reals over $V$. However, the sequence $\langle  c+a: a \in \MRB^V   \rangle$
fails to be $\MCB((2^{\aleph_0})^V)$-generic over $V$. In fact, there is no sequence $\langle  c_i: i < \omega_1 \rangle \in V[r_1, r_2]$
of Cohen reals which is $\MCB(\omega_1)$-generic over $V$.

 In this paper, we extend the above mentioned result by showing that if we force with
 $\MRB(\kappa) \times \MRB(\k)$,  then in the resulting extension, we can find a sequence
 $\langle c_i: i<\k  \rangle$ of reals which is  $\MCB(\kappa)$-generic over the ground model:
 \begin{theorem}
 \label{main theorem}
 Let $\kappa$ be an infinite cardinal. Then, forcing with $\MRB(\kappa) \times \MRB(\k)$ adds a generic filter for $\MCB(\kappa).$
 \end{theorem}
In Section \ref{forcings}, we briefly review the forcing notions  $\MCB(\k)$ and $\MRB(\k)$. Then in Section \ref{steinhaus}, we state some results from analysis which are needed for the proof of above theorem and in Section \ref{proof of main theorem}, we give a proof of Theorem \ref{main theorem}.

\section{Cohen and Random  forcings}
\label{forcings}
In this section we briefly review the forcing notions  $\MCB(\k)$ and $\MRB(\k)$, and present some of their properties.
\subsection{Cohen forcing}
Let $I$ be a non-empty set. The forcing notion $\MCB(I)$, the Cohen forcing for adding $|I|$-many Cohen
reals is defined by
\begin{center}
$\MCB(I)=\{p: I \times \omega \to 2: |p|< \aleph_0                      \}$,
\end{center}
which is ordered by reverse inclusion.
\begin{lemma}
\label{chain condition lemma for cohen}
 $\MCB(I)$ is c.c.c.
\end{lemma}
Assume $G$ is $\MCB(I)$-generic over $V$, and set $F=\bigcup G: I \times \omega \to 2.$ For each $i \in I$ set $c_i: \omega \to 2$
be defined by $c_i(n)=F(i, n).$ Then:
\begin{lemma}
For  each $i \in I, c_i \in 2^\omega$
is a new real and for $i \neq j$ in $I, c_i \neq c_j$. Further, $V[G]=V[\langle  c_i: i \in I    \rangle].$
\end{lemma}
The reals $c_i$ are called Cohen reals. By $\kappa$-Cohen reals over $V$, we mean a sequence  $\langle c_i: i< \kappa  \rangle $
which is $\MCB(\kappa)$-generic over $V$.

\subsection{Random forcing}
In this subsection we briefly review random forcing.
Suppose $I$ is a non-empty set and consider the product measure space $2^{I \times \omega}$ with the standard product measure $\mu_I$ on it.  Let $\mathbb{B}(I)$ denote the class of Borel subsets of $2^{I \times \omega}$.
Recall that $\mathbb{B}(I)$ is the $\sigma$-algebra generated by the basic open sets
\[
[s_p]=\{ x \in  2^{I \times \omega}: x \supseteq p              \},
\]
where $p \in \MCB(I)$. Also $\mu_I([s_p])= 2^{-|p|}$.

For Borel sets $S, T \in \mathbb{B}(I)$ set
\[
S \sim T \iff S \bigtriangleup T \text{~is null,~}
\]
where $S \bigtriangleup T$ denotes the symmetric difference of $S$ and $T$. The relation $\sim$ is easily seen to be an equivalence relation
on $\mathbb{B}(I).$
Then $\MRB(I)$, the forcing for adding $|I|$-many random reals, is defined as

\[
\MRB(I)= \mathbb{B}(I) / \sim.
\]
Thus elements of $\MRB(I)$ are equivalent classes $[S]$ of Borel sets modulo null sets. The order relation is defined by
\[
[S] \leq [T] \iff \mu(S \setminus T) =0.
\]
The following fact is standard.
\begin{lemma}
\label{chain condition lemma for random}
$\MRB(I)$ is c.c.c.
\end{lemma}
Using the above lemma, we can easily show that $\MRB(I)$ is in fact a complete Boolean algebra.
Let $\lusim{F}$
  be an $\MRB(I)$-name for a function from $I \times \omega$ to $2$ such that for each $i \in I, n \in \omega$ and $k < 2,$ $\parallel  \lusim{F}(i, n) =k     \parallel_{\MRB(I)} = p_k^{i, n},$ where
\[
p_k^{i, n}= [x \in 2^{I \times \omega}: x(i, n)=k].
\]
This defines $\MRB(I)$-names $\lusim{r}_i \in 2^\omega, i \in I,$ such that
\[
\parallel \forall n < \omega,~\lusim{r}_i(n)= \lusim{F}(i, n)\parallel_{\MRB(I)} =1_{\MRB(I)}= [2^{I \times \omega}].
\]
\begin{lemma}
Assume $G$ is $\MRB(I)$-generic over $V$, and for each $i \in I$ set $r_i = \lusim{r}_i[G].$ Then each $r_i \in 2^\omega$
is a new real and for $i \neq j$ in $I, r_i \neq r_j$. Further, $V[G]=V[\langle  r_i: i \in I    \rangle].$
\end{lemma}
The reals $r_i$ are called random reals. By $\kappa$-random reals over $V$, we mean a sequence  $\langle r_i: i< \kappa  \rangle $
which is $\MRB(\kappa)$-generic over $V$.

\section{Some results from analysis}
\label{steinhaus}
A famous theorem of Steinhaus \cite{steinhaus} from $\text{1920}$ asserts that if $A, B \subseteq \mathbb{R}^n$ are measurable sets with positive
Lebesgue measure, then  $A + B$ has an interior point; see also \cite{stromberg}.  Here, we
need a version of Steinhaus theorem  for  the space
$2^{\k \times \omega}$.

For $S, T \subseteq 2^{\k \times \omega}$, set $S+T=\{x+y: x \in S$ and $ y \in T              \}$, where
$x+y: \k \times \omega \to 2$ is defined by
$$(x+y)(\a,n) = x(\a,n)+y(\a,n)~(\text{mod} ~2).$$
Note that the above addition is continuous.
\begin{lemma}
\label{steinhaus lemma 1}
Suppose $S \subseteq 2^{\k \times \omega}$ is Borel and non-null. Then $S-S$ contains an open set around the zero function $0$.
\end{lemma}
\begin{proof}
We follow \cite{stromberg}. Set $\mu=\mu_\k$ be the product measure on $2^{\k \times \omega}$. As $S$ is Borel and non-null, there is a compact subset of $S$ of positive $\mu$-measure, so may suppose that $S$ itself is compact. Let $U \supseteq S$
be an open set with $\mu(U) < 2\cdot \mu(S).$
By continuity of addition, we can find an open set $V$ containing the zero function $0$ such that $V+S \subseteq U.$

We show that $V \subseteq S - S$. Thus suppose $x \in V.$ Then $(x+S) \cap S \neq \emptyset,$ as otherwise we will have
$(x+S) \cup S \subseteq U,$ and hence $\mu(U) \geq 2\cdot \mu(S),$ which is in contradiction with our choice of $U$.
Thus let $y_1, y_2 \in S$ be such that $x+ y_1 = y_2.$ Then $x= y_2 - y_1 \in S - S$ as required.
\end{proof}
Similarly, we have the following:
\begin{lemma}
\label{steinhaus lemma}
Suppose $S, T \subseteq 2^{\k \times \omega}$ are Borel and non-null. Then $S+T$ contains an open set.
\end{lemma}
Suppose $S, T \subseteq 2^{\k \times \omega}$ are Borel and non-null. It follows from Lemma \ref{steinhaus lemma} that for some $p \in \MCB(\k),$
$[s_p] \subseteq S+T.$ Thus, by continuity of the addition, we can find
$x \in S$ and $y \in T$ such that:
\begin{itemize}
\item $(x+y) \upharpoonright \dom(p) = p.$
\item The sets $S \cap [s_{x \upharpoonright \dom(p)}]$ and $T \cap [s_{y \upharpoonright \dom(p)}]$ are Borel and non-null.
\end{itemize}

\section{Proof of Theorem \ref{main theorem}}
\label{proof of main theorem}
In this section, we complete the proof of Theorem \ref{main theorem}. Thus force with $\MRB(\k) \times \MRB(\k)$
and let $G \times H$ be generic over $V$. Let $\langle \langle r_\a: \a < \k \rangle, \langle  s_\a: \a < \kappa   \rangle \rangle$
be the sequence of random reals added by $G \times H.$

For $\a < \k$ set $c_\a=r_\a+s_\a$. The following completes the proof:
\begin{lemma}
The sequence $\langle  c_\a: \a < \k      \rangle$ is a sequence of $\kappa$-Cohen reals over $V$.
\end{lemma}
\begin{proof}
It suffices to prove the following:

$\hspace{1.5cm}$ For every $([S], [T])\in {\MRB}(\k) \times \MRB(\k)$,
and every open dense subset $D\in V$

$(*)$ $\hspace{0.9cm}$  of ${\MCB}(\k)$,~ there is
$([\ov{S}], [\ov{T}])\leq ([S], [T])$ such~ that $([\ov{S}], [\ov{T}]) \vdash
``  \lan \lusim{c}_\a : \a\in \k\ran$

$\hspace{1.5cm}$  extends some element of $D$''.

Thus fix $([S], [T])\in {\MRB}(\k) \times \MRB(\k)$ and $D \in V$ as above, where $S, T \subseteq 2^{\k \times \omega}$ are Borel and non-null.
By Lemma \ref{steinhaus lemma} and the remarks after it, we can find
$p \in \MCB(\k)$ and $(x, y) \in S \times T$ such that:
\begin{enumerate}
\item $[s_p] \subseteq S+T.$
\item $(x+y) \upharpoonright \dom(p) = p.$
\item The sets $S \cap [s_{x \upharpoonright \dom(p)}]$ and $T \cap [s_{y \upharpoonright \dom(p)}]$ are Borel and non-null.
\end{enumerate}
Now let $q \in D$ be such that
\begin{center}
$([S \cap [s_{x \upharpoonright \dom(p)}]], [T \cap [s_{y \upharpoonright \dom(p)}]]) \Vdash$``$q \leq_{\MCB(\k)} p$''.
\end{center}
Using continuity of the addition and further application of
Lemma \ref{steinhaus lemma} and the remarks after it, we can find
 $x', y' $ such that:
\begin{enumerate}
\item [(4)] $x' \in S \cap [s_{x \upharpoonright \dom(p)}]$ and $y' \in T \cap [s_{y \upharpoonright \dom(p)}].$
\item [(5)] $(x'+y') \upharpoonright \dom(q) = q.$
\item [(6)] The sets $S \cap [s_{x' \upharpoonright \dom(q)}]$ and $T \cap [s_{y' \upharpoonright \dom(q)}]$ are Borel and non-null.
\end{enumerate}
It is now clear that
\begin{center}
$([S \cap [s_{x' \upharpoonright \dom(q)}]], [T \cap [s_{y' \upharpoonright \dom(q)}]]) \Vdash$``$ \lan \lusim{c}_\a : \a\in \k\ran$ extends $q$''.
\end{center}
The result follows.
\end{proof}




Mohammad Golshani,
School of Mathematics, Institute for Research in Fundamental Sciences (IPM), P.O. Box:
19395-5746, Tehran-Iran.

E-mail address: golshani.m@gmail.com

http://math.ipm.ac.ir/golshani/

\end{document}